\documentclass[11pt]{amsart}
\usepackage{amssymb,amsmath,amsfonts,amsthm}
\usepackage[margin=1.1in]{geometry}
\usepackage{xcolor}
\usepackage{hyperref}
\hypersetup{colorlinks=true,linkcolor=blue!55!black,citecolor=blue!55!black,
 urlcolor=blue!55!black,
 pdftitle={Bounded curvature manifolds without Euclidean isometric immersions of bounded mean curvature},
 pdfauthor={Haoxuan Cheng},
 pdfkeywords={Isometric immersions, mean curvature, injectivity radius, Jacobian equations}}
\newtheorem{theorem}{Theorem}[section]
\newtheorem{lemma}[theorem]{Lemma}
\newtheorem{proposition}[theorem]{Proposition}
\theoremstyle{remark}
\newtheorem{remark}[theorem]{Remark}
\numberwithin{equation}{section}
\newcommand{\R}{\mathbb R}
\DeclareMathOperator{\Lip}{Lip}
\DeclareMathOperator{\Rm}{Rm}
\DeclareMathOperator{\Scal}{Scal}
\DeclareMathOperator{\inj}{inj}
\DeclareMathOperator{\tr}{tr}
\newcommand{\dd}{\mathrm d}
\title[Bounded curvature and isometric immersions]{Bounded curvature manifolds without Euclidean isometric immersions of bounded mean curvature}
\author{Haoxuan Cheng}
\address{School of Mathematical Sciences, Fudan University,
Shanghai 200433, China}
\email{hxcheng25@m.fudan.edu.cn}
\subjclass[2020]{Primary 53C42; Secondary 53C20, 53C21}
\keywords{Isometric immersions, mean curvature, bounded curvature,
injectivity radius, Jacobian equations}
\date{}
\begin{document}
\begin{abstract}
For every integer $n\ge2$, we construct a smooth complete Riemannian
metric $G_n$ on $\R^n$ with full curvature norm at most one and
injectivity radius at least one for which no $C^2$ isometric immersion
into a finite-dimensional Euclidean space has bounded mean curvature.
In dimension two, bounded second fundamental form would give uniformly
controlled finite Jacobian representations of the Laplacian of the
conformal factor. We construct disjoint conformal blocks for which the
curvature remains bounded while the finite Jacobian representation cost
tends to infinity. Taking Euclidean products gives all higher dimensions.
The result answers Yau's Problem~52 negatively under the stronger
assumption of bounded full curvature.
\end{abstract}
\maketitle
\section{Introduction}

Bounds on intrinsic curvature do not in themselves provide uniform
control of the extrinsic geometry of a Euclidean isometric immersion.
This paper gives examples for which the full Riemann tensor and the
injectivity radius are uniformly controlled, but the mean curvature of
every finite-dimensional Euclidean realization is unbounded.

Write $\delta_n$ for the Euclidean metric on $\R^n$. We take
Hilbert--Schmidt norms of the Riemann tensor $\Rm$ and of the
normal-vector-valued second fundamental form $A$. For an isometric
immersion $F:(M^n,g)\to\R^N$, the mean curvature vector is
$H=n^{-1}\tr_g A$, and
\[
 \inj(M,g)=\inf_{p\in M}\inj_g(p).
\]

\begin{theorem}\label{thm:jac-main}
For every integer $n\ge2$, there is a smooth complete Riemannian metric
$G_n$ on $\R^n$ such that
\[
 0.9\delta_n\le G_n\le1.1\delta_n,\qquad
 \sup_{\R^n}|\Rm_{G_n}|_{G_n}\le1,\qquad
 \inj(\R^n,G_n)\ge1,
\]
but no $C^2$ isometric immersion of $(\R^n,G_n)$ into any
finite-dimensional Euclidean space has bounded mean curvature.
In particular, there is no such immersion with bounded second fundamental
form, and no such isometric embedding.
\end{theorem}

The metric $G_n$ is fixed before the target dimension is chosen. Thus one
metric in each source dimension excludes every finite-dimensional target.
Nash's theorem still supplies a smooth Euclidean isometric embedding
\cite[Theorem~3]{Nash1956}; see also Greene's treatment of noncompact
manifolds~\cite{Greene1970}. The theorem says that the mean curvature of each
such embedding is unbounded on the source.

The result is related to a question of Yau, but its intrinsic assumptions
are stronger than those in the original question. Problem~52 in Yau's
1982 problem section asks whether bounded Ricci curvature and a positive
injectivity-radius lower bound suffice for a Euclidean isometric embedding
with bounded mean curvature \cite[p.~681]{Yau1982}. Since a bound on
$\Rm$ implies a bound on the Ricci tensor, Theorem~\ref{thm:jac-main}
answers that question negatively. It also applies to Yau's later
formulations with only a Ricci lower bound
\cite[pp.~244--245]{Yau2000}; see also \cite[Section~7]{Yau2012}.

For a metric with unbounded full curvature, the uncontracted Gauss equation
immediately rules out a bounded second fundamental form. That direct
obstruction is unavailable here because $|\Rm|\le1$. Instead, the argument
uses information not contained in pointwise curvature bounds: the complexity
of a Jacobian representation of the conformal curvature density.

Consider a square $Q\subset\R^2$ and a conformal metric
\[
 g=e^{2u}(\dd x^2+\dd y^2).
\]
With the Euclidean Laplacian
$\Delta=\partial_x^2+\partial_y^2$, its Gaussian curvature is determined by
\[
 \Delta u=-e^{2u}K_g.
\]
Suppose that $F:(Q,g)\to\R^N$ is a $C^2$ isometric immersion. The fields
\[
 E_1=e^{-u}F_x,\qquad E_2=e^{-u}F_y
\]
form an orthonormal tangent frame. The classical moving-frame identity
\cite[Chapter~5]{Helein2002}
\cite[equation~(1.7) and footnote~6]{Riviere2010} gives
\begin{equation}\label{eq:jac-intro-frame}
 \Delta u=\sum_{\alpha=1}^N
 J(E_2^\alpha,E_1^\alpha),\qquad
 J(a,b)=a_xb_y-a_yb_x.
\end{equation}
Jacobian equations of this kind are the source of Wente-type estimates and
also appear in compensated compactness~\cite{Wente1969,CLMS1993}.
The scalar function $\Delta u$ is also the curvature density, since
$K_g\,\dd\operatorname{vol}_g=-\Delta u\,\dd x\,\dd y$.

For a finite representation
$\rho=\sum_iJ(a_i,b_i)$ with $a_i,b_i\in\Lip(Q)$, its representation size is
\[
 \sum_i\Lip(a_i)\Lip(b_i).
\]
Section~\ref{sec:jac-analytic} defines the cost of $\rho$ as the infimum of
these sizes. Every function smooth on a neighborhood of the closed square
has finite cost, so mere representability is not an obstruction. The
relevant fact is that the cost cannot be bounded in terms of
$\|\rho\|_\infty$ alone.

We choose smooth densities $\rho_j$ with uniformly bounded amplitudes and
costs tending to infinity. Poisson potentials of these densities are cut
off and placed in pairwise disjoint planar regions. This produces a global
conformal factor whose value, gradient, and Laplacian are uniformly
bounded. Consequently the metric is complete and has bounded Gaussian
curvature. Jacobi-field and parallel-transport estimates give the required
injectivity-radius bound; see Section~\ref{sec:jac-metric}.

Now assume that this metric has a finite-dimensional isometric immersion
with bounded second fundamental form. The Gauss formula bounds the
Euclidean Lipschitz constants of the frame fields in
\eqref{eq:jac-intro-frame}. On every block, that identity then gives a
Jacobian representation of $\rho_j$ with one bound for its cost,
independent of $j$. This contradicts the choice of the densities. The
scalar Gauss equation
\[
 \Scal_g=n^2|H|^2-|A|^2
\]
converts bounded mean curvature into bounded second fundamental form
because the scalar curvature is bounded below
\cite[Theorem~8.5]{Lee2018}. Euclidean products preserve the obstruction
and prove the theorem in every dimension $n\ge2$.

The analytic construction in Section~\ref{sec:jac-analytic} is a finite-sum
version of the prescribed-Jacobian obstruction in
\cite[Proposition~1.5]{Takac2025}. It follows the finite-tree strategy in
\cite{BuragoKleiner1998,DKK2018,Takac2025}, together with a boundary estimate
that is uniform in the number of summands. McMullen's construction provides
related background for the single-Jacobian problem \cite{McMullen1998}. The
required two-dimensional bilipschitz dichotomy is proved directly in
Appendix~\ref{app:jac-dichotomy}. Dacorogna--Moser's existence theorem for
prescribed Jacobians \cite[Theorem~1]{DacorognaMoser1990} does not give a
cost bound from the amplitude alone.

The distinction between zeroth-order bounded geometry and derivative
bounds is important here. We use the former term for bounded full curvature
together with a positive injectivity-radius lower bound. The examples have
unbounded first curvature derivative. Petrunin's tubed-embedding theorem
instead assumes uniform bounds on every covariant derivative of curvature
\cite[Section~2]{Petrunin2024}. Section~\ref{sec:jac-local-scope} records the
curvature-derivative calculation and the corresponding failure of a
uniform fixed-radius estimate.

Section~\ref{sec:jac-analytic} proves the analytic obstruction.
Section~\ref{sec:jac-metric} constructs the planar metric, and
Section~\ref{sec:jac-immersion} proves the nonimmersion statement and its
extension by products. Section~\ref{sec:jac-local-scope} discusses local
estimates and the role of curvature derivatives.

\section{A finite-sum Jacobian obstruction}\label{sec:jac-analytic}

All Lipschitz constants in this section refer to Euclidean distances.
Set $Q_0=[0,1]^2$. Equalities between Jacobians of Lipschitz maps are
understood almost everywhere, using Rademacher's theorem
\cite[Section~3.1.2]{EvansGariepy1992}.

For $\rho\in L^\infty(Q_0)$, define its finite Jacobian representation
cost by
\[
 \mathcal C_{Q_0}(\rho)=
 \inf\left\{\sum_{i=1}^m\Lip(a_i)\Lip(b_i):
 \rho=\sum_{i=1}^mJ(a_i,b_i),\quad
 a_i,b_i\in\Lip(Q_0),\quad m<\infty\right\},
\]
with $\inf\varnothing=+\infty$. This cost refers to the fixed Euclidean
coordinates. For every $\rho$ smooth on a neighborhood of $Q_0$, set
\[
 P(x,y)=\int_0^x\rho(s,y)\,\dd s,\qquad T(x,y)=y.
\]
These functions are Lipschitz on $Q_0$ and satisfy $J(P,T)=\rho$.
Thus smooth densities always have a finite-cost representation.
The obstruction below is the absence of a bound for this cost in terms
of the amplitude alone.

The following is the finite, unweighted form of the obstruction stated
in~\cite[Proposition~1.5]{Takac2025}: take source dimension two and
constant weights equal to one. The proof below includes the smooth
extension with fixed support needed in Section~\ref{sec:jac-metric}.

\begin{proposition}\label{prop:jac-density}
For every $S>0$, there exists
$\widetilde\rho\in C_c^\infty((-1,2)^2)$ with
$\|\widetilde\rho\|_\infty\le2$ such that every representation
\[
 \widetilde\rho|_{Q_0}=\sum_{i=1}^m J(a_i,b_i),
 \qquad a_i,b_i\in\Lip(Q_0),\quad m<\infty,
\]
satisfies
\[
 \sum_{i=1}^m\Lip(a_i)\Lip(b_i)>S.
\]
\end{proposition}

In particular, the proposition gives
$\mathcal C_{Q_0}(\widetilde\rho|_{Q_0})\ge S$.
No attainment of the infimum is asserted; the proof uses the strict
lower bound for each actual representation.

The next boundary estimate controls all summands through the sum of
their squared Lipschitz constants. It follows from Stokes' formula and
Cauchy--Schwarz, with an approximation argument for Lipschitz maps.

\begin{lemma}\label{lem:jac-boundary}
Let $Q,Q'=Q+\tau$ be adjacent squares of side length $r$.
Suppose $\pi_i=(a_i,b_i)$, $1\le i\le m$, are Lipschitz maps on
$Q\cup Q'$, both scalar components having Lipschitz constant at most $L_i$,
and $\sum_iL_i^2\le S$. Let $W_i\in\R^2$ be constant vectors. If
\[
 \sup_{x\in Q}\sum_i
 |\pi_i(x+\tau)-\pi_i(x)-W_i|^2\le r^2\varepsilon^2,
\]
then, for $\sigma=\sum_iJ(a_i,b_i)$,
\[
 \left|\int_Q\sigma-\int_{Q'}\sigma\right|
 \le4\sqrt{2S}\,r^2\varepsilon.
\]
\end{lemma}

\begin{proof}
For two Lipschitz maps $\pi=(a,b)$ and $\widehat\pi=(c,d)$ on $Q$,
Stokes' formula and integration by parts on the closed boundary give
\begin{equation}\label{eq:jac-stokes}
 \int_Q[J(a,b)-J(c,d)]
 =\int_{\partial Q}(a-c)\,\dd b-(b-d)\,\dd c.
\end{equation}
The boundary is oriented counterclockwise, and derivatives there are
tangential weak derivatives along its four edges.

For clarity, the formula remains valid at the stated regularity.
Extend each scalar function to the plane without increasing its Lipschitz
constant, by the scalar extension theorem
\cite[Section~3.1.1, Theorem~1]{EvansGariepy1992}, and mollify.
The approximations converge uniformly and have
uniformly bounded first derivatives. The identity
\[
 \int\psi J(a,b)=-\int a(\psi_xb_y-\psi_yb_x)
 \qquad(\psi\in C_c^\infty)
\]
shows convergence of their Jacobians in distributions. The uniform
$L^\infty$ bound on the Jacobians extends this convergence to $L^1$
test functions, including the indicator of $Q$.
Along each edge, uniform convergence and the uniform Lipschitz bound imply
weak-star convergence of tangential derivatives. Consequently the boundary
integrals also converge. Passing to the limit in the smooth Stokes formula
proves~\eqref{eq:jac-stokes}.

Apply this formula with
$\widehat\pi_i(x)=\pi_i(x+\tau)-W_i$ and
$\xi_i=\pi_i-\widehat\pi_i$. Translation does not change a Jacobian, and each
scalar component of $\widehat\pi_i$ has Lipschitz constant at most $L_i$.
Thus
\begin{align*}
 \left|\int_Q\sigma-\int_{Q'}\sigma\right|
 &\le\int_{\partial Q}\sum_iL_i(|\xi_i^1|+|\xi_i^2|)\,\dd s\\
 &\le\sqrt{2S}\int_{\partial Q}
       \left(\sum_i|\xi_i|^2\right)^{1/2}\dd s\\
 &\le4\sqrt{2S}\,r^2\varepsilon.
\end{align*}
The Cauchy--Schwarz inequality is applied inside the boundary integral.
\end{proof}

The following two-dimensional special case of
\cite[Lemma~3.3]{DKK2018} is proved directly in
Appendix~\ref{app:jac-dichotomy}. Here $L$-bilipschitz means that distances are
bounded below by $1/L$ and above by $L$ times the source distance.
We write $\mathbf e_1=(1,0)$.

\begin{lemma}\label{lem:jac-dichotomy}
For every $L\ge1$ and $0<\varepsilon<1$, there are
integers $M\ge1$, $K_0\ge2$ and $\varphi\in(0,1)$ such that, for every
integer $K\ge K_0$, $c>0$, and finite integer $N\ge2$, any
$L$-bilipschitz map
\[
 h:[0,c]\times[0,c/K]\longrightarrow\R^N
\]
satisfies at least one of the following alternatives. Put $r=c/K$.
First, for some $i\in\{0,\ldots,K-2\}$,
\begin{equation}\label{eq:jac-dichotomy-good}
 \left|h(x+r\mathbf e_1)-h(x)-\frac{h(c\mathbf e_1)-h(0)}K\right|
 \le r\varepsilon
 \quad\text{for all }x\in[ir,(i+1)r]\times[0,r].
\end{equation}
Second, there is a point
\[
 z\in\frac c{KM}\mathbb Z^2
 \cap\bigl([0,c-c/(KM)]\times[0,c/K-c/(KM)]\bigr)
\]
such that
\begin{equation}\label{eq:jac-dichotomy-growth}
 \frac{|h(z+c\mathbf e_1/(KM))-h(z)|}{c/(KM)}
 >(1+\varphi)\frac{|h(c\mathbf e_1)-h(0)|}c.
\end{equation}
The constants depend only on $L$ and $\varepsilon$, and may be chosen
as in~\eqref{eq:jac-explicit-parameters}.
\end{lemma}

The cited result treats general source dimension and gives a proportion
of good adjacent pairs. Here one pair in source dimension two suffices.

The proof adapts the finite-tree strategy of the prescribed-Jacobian
constructions in~\cite{BuragoKleiner1998,DKK2018,Takac2025}.
We keep the tree finite and use Lemma~\ref{lem:jac-boundary} to control
the entire sum, independently of its length.

\begin{proof}[Proof of Proposition~\ref{prop:jac-density}]
Fix $S>0$ and set
\[
 L=\sqrt{1+2S},\qquad
 \varepsilon=\min\{1/4,1/(16\sqrt{2S})\}.
\]
Choose $M,K_0,\varphi$ from Lemma~\ref{lem:jac-dichotomy}. Choose an integer
$k_0\ge1$ with $(1+\varphi)^{k_0}>L^2$, and then an integer
$K\ge\max\{K_0,16,2\}$. All choices precede the choice of any Jacobian
representation.

We construct a finite tree of thin rectangles inside $Q_0$.
Its root is $R_0=[0,1]\times[0,1/K]$.
For a rectangle $R=p+[0,c]\times[0,c/K]$, put $r=c/K$ and partition it
into the $K$ squares
\[
 Q_i=q_i+[0,r]^2,\qquad q_i=p+(ir,0),\quad0\le i<K.
\]
Inside each $Q_i$, place the $M^2$ child rectangles
\begin{equation}\label{eq:jac-children}
 R_{i,s,t}=q_i+\frac rM(s,t)
       +[0,r/M]\times[0,r/(KM)],\qquad0\le s,t<M.
\end{equation}
These lie in distinct cells of the $M$-by-$M$ subdivision of $Q_i$,
have disjoint interiors, and have total area $r^2/K$ in that square.
Each has the same aspect ratio $K$. Repeat the construction through
level $k_0$. At level $k$, the long side is $c_k=(KM)^{-k}$ and the
square side is $r_k=c_k/K$.

Every growth alternative~\eqref{eq:jac-dichotomy-growth} selects a child
in this tree. Indeed, relative to the lower-left vertex of $R$, write the
selected grid point as $(r/M)(u,v)$, where
$0\le u<KM$ and $0\le v<M$. Writing $u=iM+s$ identifies precisely the
child $R_{i,s,v}$. Its horizontal endpoints are the two points in
\eqref{eq:jac-dichotomy-growth}.

For an $L$-bilipschitz map $h:Q_0\to\R^N$, define the endpoint stretch
of $R=p+[0,c]\times[0,c/K]$ by
\[
 T_h(R)=\frac{|h(p+c\mathbf e_1)-h(p)|}c\in[1/L,L].
\]
Apply the dichotomy at the root and descend whenever only the growth
alternative holds. If this continued for $k_0$ steps, it would give
$T_h(R)>(1+\varphi)^{k_0}/L>L$ at the last rectangle, a contradiction.
Thus some rectangle at a level $k<k_0$ has an adjacent pair of squares
satisfying~\eqref{eq:jac-dichotomy-good} after translation.

We next prescribe a density on the entire tree. Set it equal to zero
outside $R_0$. At level zero, assign the values $1,2,1,2,\ldots$ to the
$K$ root squares. At each successive level through $k_0$, overwrite the
density on each rectangle by the same alternating values on its own
$K$ squares. Values on boundaries may be chosen arbitrarily.
Let $\rho_*$ be the final density; then $0\le\rho_*\le2$.

Consider adjacent squares $Q,Q'$ of side $r$ in any rectangle of the
tree. Immediately after that level is assigned, the absolute difference
of their integrals is $r^2$. Subsequent changes occur only in the union of
the next-level rectangles. This union has area at most $r^2/K$ in each
square, and the final value differs from the earlier value by at most one.
Consequently
\begin{equation}\label{eq:jac-density-gap}
 \left|\int_Q\rho_* -\int_{Q'}\rho_*\right|
 \ge(1-2/K)r^2\ge\tfrac78r^2>\tfrac34r^2.
\end{equation}
This estimate compares the final value with the value at the chosen
level; it does not sum an error over subsequent levels.

Let $\mathcal U\subset L^\infty(Q_0)$ consist of the functions for which
the last strict inequality in~\eqref{eq:jac-density-gap} holds for every
adjacent pair in the tree. The tree is finite, and square indicators
belong to $L^1(Q_0)$. Hence $\mathcal U$ is weak-star open and contains
$\rho_*$. Extend $\rho_*$ by zero to the plane and convolve with a
nonnegative unit-mass mollifier of radius $\delta$. As $\delta\downarrow0$,
the restrictions converge weak-star to $\rho_*$ on $Q_0$ and remain bounded
between zero and two. Choose $0<\delta<1/2$ sufficiently small that the
restriction belongs to $\mathcal U$. The resulting function
$\widetilde\rho$ is smooth, compactly supported in $(-1,2)^2$, and has
the required amplitude bound.

Suppose its restriction had a finite representation of cost at most $S$.
Delete zero terms with a constant component. Rescale the two components
of each remaining term by reciprocal positive constants so that
\[
 \Lip(a_i)=\Lip(b_i)=L_i,
 \qquad\sum_iL_i^2\le S.
\]
Explicitly, multiply $a_i$ by
$\sqrt{\Lip(b_i)/\Lip(a_i)}$ and $b_i$ by its reciprocal. The Jacobian
does not change. Following the identity-adjoining device
in~\cite[p.~8]{Takac2025}, set $\pi_i=(a_i,b_i)$ and
\[
 h(x)=(x,\pi_1(x),\ldots,\pi_m(x))\in\R^{2+2m}.
\]
Then
\[
 |x-y|^2\le|h(x)-h(y)|^2\le(1+2S)|x-y|^2,
\]
so $h$ is $L$-bilipschitz. If all terms were deleted, use $h(x)=x$.
The preceding tree argument supplies a rectangle
$R=p+[0,c]\times[0,c/K]$ and adjacent squares $Q,Q'=Q+r\mathbf e_1$ for which
\eqref{eq:jac-dichotomy-good} holds. Projecting to the $\pi_i$ components
and putting $W_i=K^{-1}(\pi_i(p+c\mathbf e_1)-\pi_i(p))$ gives
\[
 \sup_{x\in Q}\sum_i
 |\pi_i(x+r\mathbf e_1)-\pi_i(x)-W_i|^2\le r^2\varepsilon^2.
\]
Lemma~\ref{lem:jac-boundary} therefore bounds the difference of the two
density integrals by $4\sqrt{2S}\,r^2\varepsilon\le r^2/4$.
Membership in $\mathcal U$ requires that difference to exceed $3r^2/4$.
This contradiction proves the proposition.
\end{proof}

\section{A complete conformal metric with bounded curvature}
\label{sec:jac-metric}

This section converts the analytic densities into a single complete
metric. The potential estimates use the classical planar fundamental
solution~\cite[Section~2.2.1]{Evans1998}; the cutoff, placement of the
blocks, and uniform injectivity estimate are given explicitly below.
We prescribe the metric directly in global conformal coordinates.
For coordinate regularity of a given metric, see
\cite[Sections~2 and~4]{DeTurckKazdan1981}.

Apply Proposition~\ref{prop:jac-density} with $S=j$ for each integer
$j\ge1$. This gives functions
\[
 \widetilde\rho_j\in C_c^\infty((-1,2)^2),\qquad
 \|\widetilde\rho_j\|_\infty\le2,
 \qquad \rho_j=\widetilde\rho_j|_{Q_0},
\]
such that every finite Jacobian representation of $\rho_j$ has cost
greater than $j$. Only the amplitudes and supports are uniformly
controlled; no bounds on derivatives of these densities are imposed.

Let $\Phi(x)=(2\pi)^{-1}\log|x|$, so that $\Delta\Phi=\delta_0$
in distributions, and set $v_j=\Phi*\widetilde\rho_j$.
These functions are smooth and satisfy $\Delta v_j=\widetilde\rho_j$.
On $\Omega=(-2,3)^2$, they obey the uniform estimates
\begin{equation}\label{eq:jac-potential}
 \|v_j\|_{L^\infty(\Omega)}\le C_0,
 \qquad \|Dv_j\|_{L^\infty(\Omega)}\le C_1,
\end{equation}
where one may take
\[
 C_0=\frac1\pi\int_{B(0,6)}|\log|z||\,\dd z,
 \qquad C_1=\frac1\pi\int_{B(0,6)}|z|^{-1}\,\dd z.
\]
Indeed, $|x-y|<6$ for $x\in\Omega$ and
$y\in\operatorname{supp}\widetilde\rho_j$, and both kernels in these
integrals are locally integrable in dimension two.

Choose a fixed cutoff $0\le\chi\le1$, with
$\chi\in C_c^\infty(\Omega)$ and $\chi=1$ on $[-1,2]^2$.
Put $M_1=\|D\chi\|_\infty$, $M_2=\|\Delta\chi\|_\infty$, and fix
\begin{align*}
 C&=\max\{1,C_0,C_1+M_1C_0,2+2M_1C_1+M_2C_0\},\\
 a&=(1000C)^{-1},\qquad b=Ca=10^{-3}.
\end{align*}
The functions $u_j=a\chi v_j$, extended by zero outside $\Omega$, are
smooth. The product rule and~\eqref{eq:jac-potential} give
\begin{equation}\label{eq:jac-block}
 \|u_j\|_\infty\le b,\qquad
 \|Du_j\|_\infty\le b,\qquad
 \|\Delta u_j\|_\infty\le b,
 \qquad \Delta u_j=a\rho_j\quad\text{on }Q_0.
\end{equation}
In particular, the positive parameter $a$ is fixed independently of $j$.

For $p_j=(10j,0)$, define
\begin{equation}\label{eq:jac-global}
 U(x)=\sum_{j=1}^\infty u_j(x-p_j),\qquad
 g=e^{2U}(\dd x^2+\dd y^2).
\end{equation}
The supports are pairwise disjoint and locally finite. Hence $U$ is
smooth, and each of $\|U\|_\infty$, $\|DU\|_\infty$, and
$\|\Delta U\|_\infty$ is at most $b$. Consequently
\begin{equation}\label{eq:jac-intrinsic}
 e^{-2b}\delta_2\le g\le e^{2b}\delta_2,
 \qquad K_g=-e^{-2U}\Delta U,
 \qquad |\Rm_g|_g=2|K_g|\le2e^{2b}b<1.
\end{equation}
The first comparison implies $0.9\delta_2\le g\le1.1\delta_2$ and
completeness: the induced distance is uniformly equivalent to the complete
Euclidean distance.

The next proof is a quantitative application of parallel transport,
the Jacobi equation, and Hopf--Rinow; see
\cite[Theorems~6.19 and~10.1, Corollary~6.21]{Lee2018}.
We prove the uniform estimate for these particular metrics.

\begin{proposition}\label{prop:jac-inj}
The metric~\eqref{eq:jac-global} satisfies $\inj(\R^2,g)\ge1$.
\end{proposition}

\begin{proof}
All tangent spaces are identified with $\R^2$ using the global coordinates.
For Euclidean vectors $v,w$, the Christoffel bilinear form is
\[
 \Gamma(v,w)=\langle DU,v\rangle w+
 \langle DU,w\rangle v-\langle v,w\rangle DU,
 \qquad |\Gamma(v,w)|_\delta\le3b|v|_\delta|w|_\delta.
\]
Fix $p$ and $v\in T_p\R^2$ with $|v|_\delta<2$. The geodesic
$\gamma(t)=\exp_p(tv)$, $0\le t\le1$, exists by completeness.
Its constant $g$-speed is $\ell=|v|_{g_p}\le2e^b$, and
$|\gamma'(t)|_\delta\le e^b\ell\le2e^{2b}$.

Vary the initial velocity through $v+sw$. The resulting Jacobi field
$\mathcal J$ satisfies
\[
 \mathcal J(0)=0,\qquad D_t\mathcal J(0)=w,\qquad \mathcal J(1)=D\exp_p(v)w.
\]
Let $\mathcal P(t)$ be parallel transport along $\gamma$ and put $Z(t)=\mathcal P(t)^{-1}\mathcal J(t)$.
The Jacobi equation, integrated twice in the fixed space $(T_p\R^2,g_p)$,
becomes
\begin{equation}\label{eq:jac-volterra}
 Z(t)=tw-\int_0^t(t-s)\mathcal R(s)Z(s)\,\dd s,
 \qquad
 \|\mathcal R(s)\|_{g_p}\le\ell^2|\Rm_g|_g
 \le8e^{4b}b=:\kappa.
\end{equation}
Since $\kappa\le16b<1$, the quantity
$M_Z=\sup_{0\le t\le1}|Z(t)|_{g_p}$ satisfies
$M_Z\le|w|_{g_p}+(\kappa/2)M_Z$, so $M_Z\le2|w|_{g_p}$.
Equation~\eqref{eq:jac-volterra} then gives
$|Z(1)-w|_{g_p}\le\kappa|w|_{g_p}$.
Thus the map $\mathcal Q_v:w\mapsto Z(1)$ obeys
$\|\mathcal Q_v-I\|_\delta\le\kappa$; the operator norms for $g_p$ and
$\delta$ agree because $g_p$ is a scalar multiple of $\delta$.

In coordinates, $\mathcal P'=-\mathsf B(t)\mathcal P$, $\mathcal P(0)=I$, and
\[
 \|\mathsf B(t)\|_\delta\le3b|\gamma'(t)|_\delta
 \le6e^{2b}b=:\Lambda\le12b.
\]
Gronwall's inequality gives $\|\mathcal P(1)\|_\delta\le e^{\Lambda}$ and
$\|\mathcal P(1)-I\|_\delta\le e^{\Lambda}-1$. Since $D\exp_p(v)=\mathcal P(1)\mathcal Q_v$,
\begin{equation}\label{eq:jac-dexp}
 \|D\exp_p(v)-I\|_\delta
 \le e^{\Lambda}\kappa+e^{\Lambda}-1
 \le2(16b)+2(12b)=56b<\tfrac12.
\end{equation}
Here $\Lambda\le0.012$ allows $e^{\Lambda}\le2$ and $e^{\Lambda}-1\le2\Lambda$.
At $v=0$, the differential is the identity.

Integrating~\eqref{eq:jac-dexp} along a straight segment in the convex
Euclidean tangent ball $B_\delta(0,2)$ yields
\[
 |\exp_p(v)-\exp_p(w)-(v-w)|_\delta
 \le\tfrac12|v-w|_\delta.
\]
The exponential map is therefore injective on this ball, and its
differential is nonsingular there. The $g_p$-unit ball lies inside it,
since $|v|_\delta\le e^b|v|_{g_p}$.

For $|v|_{g_p}<1$, let $q=\exp_p(v)$. A minimizing geodesic from $p$
to $q$ exists by completeness. Its initial velocity $w$, with affine
parameter interval $[0,1]$, satisfies
\[
 \exp_p(w)=q,\qquad |w|_{g_p}=d_g(p,q)\le|v|_{g_p}<1.
\]
Injectivity in the coordinate tangent ball gives $w=v$. Thus every radial
geodesic in the $g_p$-unit ball is minimizing. Conversely, every point
of $B_g(p,1)$ is reached by such a minimizing geodesic. It follows that
$\exp_p$ maps the $g_p$-unit ball diffeomorphically onto $B_g(p,1)$.
All estimates are independent of $p$, proving the proposition.
\end{proof}

\section{The immersion obstruction and products}
\label{sec:jac-immersion}

The identity in the next lemma is classical; compare
\cite[equation~(1.7) and footnote~6]{Riviere2010} and the moving-frame
background in~\cite[Chapter~5]{Helein2002}. We include its derivation
for $C^2$ immersions and the quantitative Lipschitz estimate used here.

\begin{lemma}\label{lem:jac-frame}
Let $u$ be smooth on a neighborhood of $Q_0$, and let $F$ be a $C^2$
isometric immersion of that neighborhood, equipped with $e^{2u}\delta_2$,
into $\R^N$.
Set $E_1=e^{-u}F_x$ and $E_2=e^{-u}F_y$. Then
\begin{equation}\label{eq:jac-frame}
 \Delta u=\sum_{\alpha=1}^N J(E_2^\alpha,E_1^\alpha).
\end{equation}
If $|u|,|Du|_\delta\le b$ and $|A_F|_g\le A_0$, then
\[
 \Lip(E_i^\alpha)\le b+e^bA_0
 \qquad(i=1,2,\ 1\le\alpha\le N).
\]
\end{lemma}

\begin{proof}
The fields $e_1=e^{-u}\partial_x$ and $e_2=e^{-u}\partial_y$ form an
orthonormal frame for $g=e^{2u}\delta_2$. Their images $E_1,E_2$ are
$C^1$ orthonormal vector fields along the immersion. The conformal
connection formulas are
\[
 \nabla_{\partial_x}e_1=-u_y e_2,
 \qquad\nabla_{\partial_y}e_1=u_xe_2.
\]
The Gauss formula therefore gives
\[
 \omega:=\langle\dd E_1,E_2\rangle=-u_y\,\dd x+u_x\,\dd y.
\]
On the other hand, $\omega=\sum_\alpha E_2^\alpha\,\dd E_1^\alpha$.
Taking exterior derivatives gives
\[
 (\Delta u)\,\dd x\wedge\dd y
 =\sum_\alpha\dd E_2^\alpha\wedge\dd E_1^\alpha,
\]
which is~\eqref{eq:jac-frame}. For $C^1$ fields this computation holds
in distributions: local mollification converges in $C^1$, so the products
on the right converge uniformly. Both sides of the final identity are
continuous. No third derivatives of $F$ are needed.

For a Euclidean unit vector $v$, the Gauss formula also gives
\[
 \dd E_1(v)=\omega(v)E_2+A_F(v,e_1).
\]
Since $|v|_g=e^u$, we obtain
$|\dd E_1(v)|\le b+e^bA_0$. The same estimate holds for $E_2$.
Integrating along line segments in the convex square $Q_0$ proves the
Lipschitz estimates for the vector fields and their scalar components.
\end{proof}

Applying the frame estimate to the metric of
Section~\ref{sec:jac-metric} turns the analytic cost bound into an
obstruction to bounded second fundamental form.

\begin{proposition}\label{prop:jac-no-A}
The metric $g$ in~\eqref{eq:jac-global} admits no $C^2$ isometric immersion
into any finite-dimensional Euclidean space with bounded second
fundamental form.
\end{proposition}

\begin{proof}
Suppose $F:(\R^2,g)\to\R^N$ were such an immersion and
$\sup|A_F|=A_0<\infty$. Restrict it to $p_j+Q_0$ and translate the
domain back to $Q_0$. The metric there is $e^{2u_j}\delta_2$.
By~\eqref{eq:jac-block} and Lemma~\ref{lem:jac-frame},
\[
 \rho_j=\sum_{\alpha=1}^N
 J(a^{-1/2}E_2^\alpha,a^{-1/2}E_1^\alpha).
\]
The cost of this actual representation is at most
$Na^{-1}(b+e^bA_0)^2$. Proposition~\ref{prop:jac-density}, with $S=j$,
therefore forces
\[
 j<\frac Na(b+e^bA_0)^2\qquad\text{for every }j\ge1.
\]
The right side is fixed, whereas $j$ is unbounded. This contradiction
proves the proposition for every finite $N$.
\end{proof}

The same argument gives a local quantitative statement. For any $C^2$
isometric immersion into $\R^N$ defined on a neighborhood of $p_j+Q_0$,
\begin{equation}\label{eq:jac-A-lower}
 \sup_{p_j+Q_0}|A_F|
 >e^{-b}\left(\sqrt{\frac{aj}{N}}-b\right).
\end{equation}
In particular, the obstruction is a lower bound valid for all immersions,
not the failure of an estimate for a particular construction.

The following is the scalar contraction of the classical Gauss equation
\cite[Theorem~8.5]{Lee2018}. The proof also supplies the approximation
argument at the $C^2$ regularity used in this paper.

\begin{lemma}\label{lem:jac-scalar}
Let $(M^n,g)$ be smooth with scalar curvature bounded below. Every $C^2$
Euclidean isometric immersion with bounded mean curvature has bounded
second fundamental form.
\end{lemma}

\begin{proof}
The scalar Gauss equation, with $H=n^{-1}\tr_g A$, is
\begin{equation}\label{eq:jac-scalar-gauss}
 \Scal_g=n^2|H|^2-|A|^2.
\end{equation}
We recall why it applies to $C^2$ immersions. Mollify the coordinate
functions of the immersion on a relatively compact coordinate ball.
On a smaller ball the resulting smooth maps remain immersions and converge
in $C^2$. Their induced metrics converge to $g$ in $C^1$; the orthogonal
normal projections and the second fundamental forms converge uniformly.
In fixed coordinates, the fully covariant curvature tensor is a linear
combination of second derivatives of the metric plus an expression
continuous in the inverse metric and its first derivatives. It therefore
converges in distributions. Passing to the limit in the smooth Gauss
equation gives the pointwise identity for the limiting immersion, since
both sides are continuous. Contracting with $g$ proves
\eqref{eq:jac-scalar-gauss}.

If $\Scal_g\ge-S_0$ and $|H|\le H_0$, this identity yields
$|A|^2\le n^2H_0^2+S_0$.
\end{proof}

\begin{proof}[Proof of Theorem~\ref{thm:jac-main}]
For $n=2$, use $g$ from~\eqref{eq:jac-global}.
Its smoothness, completeness, and metric and curvature bounds were
proved in Section~\ref{sec:jac-metric}; its injectivity radius is at least
one by Proposition~\ref{prop:jac-inj}.
Since $\Scal_g=2K_g$ is bounded below,
Lemma~\ref{lem:jac-scalar} and Proposition~\ref{prop:jac-no-A} exclude
bounded mean curvature.

For $n>2$, take the Riemannian product
\begin{equation}\label{eq:jac-product}
 G_n=g\oplus\delta_{n-2}
 =e^{2U(x_1,x_2)}(\dd x_1^2+\dd x_2^2)
       +\sum_{k=3}^n\dd x_k^2.
\end{equation}
The connection splits and all mixed curvature components vanish.
Thus $|\Rm_{G_n}|_{G_n}=|\Rm_g|_g\le1$ and
$\Scal_{G_n}=2K_g$. Completeness and the metric comparison are preserved.
The exponential map splits as
\[
 \exp^{G_n}_{(p,z)}(v,w)=(\exp^g_p(v),z+w).
\]
It is injective with nonsingular differential on the $G_n$-unit tangent
ball, because $|v|_g<1$ there. Completeness and the minimizing-geodesic
argument from Proposition~\ref{prop:jac-inj} give $\inj(G_n)\ge1$.

Suppose an immersion $F:(\R^n,G_n)\to\R^N$ had bounded mean curvature.
Lemma~\ref{lem:jac-scalar} first gives a bound for its entire second
fundamental form. The inclusion $\iota:(\R^2,g)\to(\R^n,G_n)$,
$\iota(p)=(p,0)$, is totally geodesic. Hence the second fundamental form
of $F\circ\iota$ satisfies
\[
 A_{F\circ\iota}(X,Y)=A_F(\dd\iota X,\dd\iota Y),
 \qquad |A_{F\circ\iota}|\le |A_F|\circ\iota.
\]
This contradicts Proposition~\ref{prop:jac-no-A}. If $N<n$, no isometric
immersion exists by the rank condition. The proof covers all finite
target dimensions.
\end{proof}

\begin{remark}
The metric in~\eqref{eq:jac-product} is a product, not the conformal
metric $e^{2U}\delta_n$. The latter would involve individual second
derivatives of $U$ in its higher-dimensional curvature tensor, whereas
our estimates only control $U$, $DU$, and $\Delta U$.
Likewise, the argument restricts the full second fundamental form to a
totally geodesic slice only after applying the scalar Gauss equation;
a mean-curvature bound cannot in general be restricted directly.
\end{remark}

\section{Local existence and uniform curvature estimates}
\label{sec:jac-local-scope}

Every smooth Riemannian manifold admits a smooth isometric embedding into
some finite-dimensional Euclidean space by Nash's theorem
\cite[Theorem~3]{Nash1956}. If $F$ is such an embedding and $V$ has compact
closure in the manifold, then
\[
 \sup_V|A_F|\le\max_{\overline V}|A_F|<\infty.
\]
This is a bound for a fixed embedding on a fixed region. On a complete
manifold one may take $V$ to be any ball of finite radius, since its closure
is compact. Thus even the metrics constructed here have local embeddings
with bounded second fundamental form. Theorem~\ref{thm:jac-main} says that
no choice of a global finite-dimensional immersion makes these bounds
uniform over the whole manifold, even with a bound depending on the metric.

The local lower bound~\eqref{eq:jac-A-lower} also rules out a curvature
estimate on a fixed radius that depends only on zeroth-order bounded
geometry. Here this term means a bound for the full curvature tensor and
a positive lower bound for the injectivity radius; no curvature derivative
bounds are included.
This convention is weaker than the bounded-geometry definition with
bounds for all covariant curvature derivatives used, for example, in
\cite[Definition~2.2]{Schick2001} and~\cite[Section~2]{Petrunin2024}.

\begin{proposition}\label{prop:jac-local-nonuniform}
For every $r_0>0$, there exists a smooth complete metric $g_{r_0}$ on
$\R^2$ satisfying
\[
 0.9\delta_2\le g_{r_0}\le1.1\delta_2,\qquad
 \sup|\Rm_{g_{r_0}}|\le1,\qquad \inj(g_{r_0})\ge1,
\]
with the following property. For every finite $N\ge2$ and every $C_A>0$,
there is a point $q$ such that every $C^2$ isometric immersion
$F:B_{g_{r_0}}(q,r_0)\to\R^N$ satisfies
\[
 \sup_{B_{g_{r_0}}(q,r_0)}|A_F|>C_A.
\]
\end{proposition}

\begin{proof}
Keep the densities, cutoff, and constants $a,b$ from
Section~\ref{sec:jac-metric}. Choose a fixed
\[
 0<t<\min\{1,r_0/(2e^b)\}.
\]
Replace each block by
\[
 u_{j,t}(x)=at^2(\chi v_j)(x/t),\qquad
 U_t(x)=\sum_{j\ge1}u_{j,t}(x-p_j),\qquad g_{r_0}=e^{2U_t}\delta_2.
\]
The supports remain disjoint and locally finite. The chain rule gives
\[
 \|U_t\|_\infty\le bt^2,\quad
 \|DU_t\|_\infty\le bt,\quad
 \|\Delta U_t\|_\infty\le b.
\]
All the estimates in Section~\ref{sec:jac-metric}, including
Proposition~\ref{prop:jac-inj}, apply with the same upper bound $b$.
For $q_j=p_j+t(1/2,1/2)$, straight coordinate segments show that
$p_j+tQ_0$ lies strictly inside $B_{g_{r_0}}(q_j,r_0)$, since
\[
 d_{g_{r_0}}(q_j,p_j+tx)\le e^{bt^2}t/\sqrt2<r_0
 \qquad(x\in Q_0).
\]

Suppose an immersion of this ball has $|A_F|\le C_A$. On the scaled square,
let $E_1,E_2$ be its natural unit tangent fields. The proof of
Lemma~\ref{lem:jac-frame} gives the coordinate Lipschitz bound
$bt+e^{bt^2}C_A$. Write $\widehat E_i(x)=E_i(p_j+tx)$. The frame identity
and the chain rule yield
\[
 at^2\rho_j=\sum_{\alpha=1}^N
 J(\widehat E_2^\alpha,\widehat E_1^\alpha),\qquad
 \Lip(\widehat E_i^\alpha)\le t(bt+e^{bt^2}C_A).
\]
Dividing each scalar component by $t\sqrt a$ gives a representation of
$\rho_j$ with cost at most
\[
 \frac Na(bt+e^{bt^2}C_A)^2.
\]
Choose an integer $j$ larger than this fixed quantity.
Proposition~\ref{prop:jac-density} gives a contradiction.
\end{proof}

For example, the uniform conclusion stated in
\cite[Local Compression Corollary~1.F]{Gromov2023} would, in dimension
two and with its image-size parameter equal to one, give a fixed source
radius, a fixed finite target dimension, and a fixed bound on normal
curvatures under $|\sec|\le1$ and $\inj\ge1$.
For a symmetric normal-vector-valued bilinear form, polarization shows
that a bound $|A(v,v)|\le C_A$ for unit $v$ implies $|A|\le nC_A$.
Thus that uniform conclusion is incompatible with
Proposition~\ref{prop:jac-local-nonuniform}. Ordinary smooth local
isometric embedding, with a bound depending on the region and metric,
is not excluded. No result from~\cite{Gromov2023} is used in our proof.
Gromov also studies curvature restrictions through focal radii
in~\cite[Section~5]{Gromov2022}.

\begin{remark}
Yau later asked for bounded mean curvature under a Ricci lower bound
alone~\cite[pp.~244--245]{Yau2000} and~\cite[Section~7]{Yau2012}.
These formulations should be distinguished from Problem~52 and from
a hypothesis of bounded full curvature. The metrics in
Theorem~\ref{thm:jac-main} have bounded full curvature, hence also a
Ricci lower bound, and address these later formulations as well.
\end{remark}

For comparison, Petrunin's tubed-embedding theorem assumes positive
injectivity radius and bounds for every covariant derivative of curvature;
under those hypotheses, a Euclidean isometric tubed embedding exists
exactly when volume growth is uniformly polynomial
\cite[Section~2, Main theorem]{Petrunin2024}. A tubed embedding has a
uniformly thick tubular neighborhood, a requirement stronger than a
bound for the second fundamental form. The derivative hypotheses do
not hold for our metrics, as shown below.
Compactness results such as~\cite[Theorem~1.1]{Breuning2015} start with
existing immersions subject to extrinsic curvature and volume bounds.

\begin{remark}
The original metric and its Euclidean products have unbounded first
curvature derivatives.
Indeed, put $B_j=\|\partial_y\rho_j\|_\infty$ and define
\[
 P_j(x,y)=\int_0^x\rho_j(s,y)\,\dd s,\qquad T(x,y)=y.
\]
Then $J(P_j,T)=\rho_j$, $\Lip(T)=1$, and
$\Lip(P_j)\le\sqrt{4+B_j^2}$. Proposition~\ref{prop:jac-density}
therefore gives $B_j>\sqrt{j^2-4}$ for $j>2$. On the corresponding
square of the original metric, $\rho_j=-a^{-1}e^{2u_j}K_g$, whence
\[
 B_j\le a^{-1}e^{3b}\sup_{p_j+Q_0}|\nabla_gK_g|+4b.
\]
Consequently $\sup|\nabla_gK_g|=\infty$, and hence
$\sup|\nabla_g\Rm_g|=\infty$.

Unbounded derivatives alone are not the representation obstruction used
here. For instance, if $\rho(x,y)=f(x)$ with $\|f\|_\infty\le2$, then
$J(\int_0^x f(s)\,\dd s,y)=\rho$ has cost at most two, independently
of the derivatives of $f$. The construction requires the stronger
failure of bounded-cost representations by arbitrary finite sums.
\end{remark}

\appendix
\section{A direct proof of the two-dimensional dichotomy}
\label{app:jac-dichotomy}

We give a direct proof of the two-dimensional, single-pair special case
of~\cite[Lemma~3.3]{DKK2018}, with constants independent of the finite
target dimension.

We prove Lemma~\ref{lem:jac-dichotomy} by estimating the squared
deviations of horizontal difference quotients from the bottom-edge
average. Summing these deviations over all grid rows selects one adjacent
pair of squares on which the translation estimate holds uniformly.

\begin{proof}[Proof of Lemma~\ref{lem:jac-dichotomy}]
Fix $L\ge1$ and $0<\varepsilon<1$. Choose
\begin{equation}\label{eq:jac-explicit-parameters}
 \begin{split}
 M&=\left\lceil\frac{16L}{\varepsilon}\right\rceil,
 \qquad \varphi=\frac{\varepsilon^2}{256ML^2},\\
 K_0&=\max\left\{2,
       \left\lceil\frac{1024ML^2}{\varepsilon^2}\right\rceil\right\}.
 \end{split}
\end{equation}
In particular $0<\varphi<1$. Let $K\ge K_0$ be an integer, let $c>0$,
and let $h:[0,c]\times[0,c/K]\to\R^N$ be $L$-bilipschitz.
Put
\[
 r=\frac cK,\qquad \delta=\frac rM,\qquad
 \bar v=\frac{h(c,0)-h(0,0)}c.
\]
The Lipschitz upper bound gives $|\bar v|\le L$.
Suppose that the growth alternative~\eqref{eq:jac-dichotomy-growth}
fails. For integers $0\le\ell<KM$ and $0\le k<M$, define
\[
 v_{\ell,k}=\frac{h((\ell+1)\delta,k\delta)
                         -h(\ell\delta,k\delta)}{\delta}.
\]
All these edges have basepoints in the grid allowed by that alternative,
so
\begin{equation}\label{eq:jac-grid-upper}
 |v_{\ell,k}|\le(1+\varphi)|\bar v|.
\end{equation}
For each row, telescoping gives its average difference quotient
\[
 \mu_k=\frac1{KM}\sum_{\ell=0}^{KM-1}v_{\ell,k}
       =\frac{h(c,k\delta)-h(0,k\delta)}c.
\]
Comparing the endpoints with those on the bottom edge yields
\[
 |\mu_k-\bar v|\le\frac{2Lk\delta}{c}\le\frac{2L}{K}.
\]
The Euclidean inner-product identity and~\eqref{eq:jac-grid-upper}
therefore imply
\begin{align*}
 \frac1{KM}\sum_{\ell=0}^{KM-1}|v_{\ell,k}-\bar v|^2
 &=\frac1{KM}\sum_{\ell=0}^{KM-1}|v_{\ell,k}|^2
          -2\langle\mu_k,\bar v\rangle+|\bar v|^2\\
 &\le(2\varphi+\varphi^2)|\bar v|^2
          -2\langle\mu_k-\bar v,\bar v\rangle\\
 &\le L^2\left(3\varphi+\frac4K\right)=:\beta.
\end{align*}
Our parameter choices ensure that
\begin{equation}\label{eq:jac-grid-energy}
 \beta\le\frac{3\varepsilon^2}{256M}
             +\frac{\varepsilon^2}{256M}
       =\frac{\varepsilon^2}{64M}.
\end{equation}

For $0\le i\le K-2$, write
$Q_i=[ir,(i+1)r]\times[0,r]$ and define the total squared deviation
over the two adjacent squares by
\[
 \mathcal E_i=\sum_{k=0}^{M-1}
       \sum_{\ell=iM}^{(i+2)M-1}|v_{\ell,k}-\bar v|^2.
\]
Each fine edge is counted at most twice in $\sum_i\mathcal E_i$.
Summing the row estimates gives
\[
 \sum_{i=0}^{K-2}\mathcal E_i
 \le2\sum_{k=0}^{M-1}\sum_{\ell=0}^{KM-1}|v_{\ell,k}-\bar v|^2
 \le2KM^2\beta.
\]
Hence a single index $i$ satisfies
\begin{equation}\label{eq:jac-pair-energy}
 \mathcal E_i\le\frac{2K}{K-1}M^2\beta
       \le4M^2\beta\le\frac{M\varepsilon^2}{16}.
\end{equation}
Fix this index. For every grid point $z=(\ell\delta,k\delta)$ with
$iM\le\ell\le(i+1)M$ and $0\le k<M$, telescoping and
Cauchy--Schwarz give
\begin{align*}
 |h(z+r\mathbf e_1)-h(z)-r\bar v|^2
 &=\left|\delta\sum_{q=\ell}^{\ell+M-1}(v_{q,k}-\bar v)\right|^2\\
 &\le\delta^2M\sum_{q=\ell}^{\ell+M-1}|v_{q,k}-\bar v|^2\\
 &\le\delta^2M\mathcal E_i
 \le\frac{r^2\varepsilon^2}{16}.
\end{align*}
Even when $\ell=(i+1)M$, the last edge has index $(i+2)M-1$,
which is included in $\mathcal E_i$.

For any $x\in Q_i$, choose one of these grid points $z$ with
$|x-z|\le\sqrt2\delta$. On the top edge $x_2=r$, use the last
allowed row $z_2=(M-1)\delta$; no estimate on fine edges in the top
row is required. The function
\[
 \mathcal T(x)=h(x+r\mathbf e_1)-h(x)-r\bar v,\qquad x\in Q_i,
\]
is $2L$-Lipschitz. Consequently
\[
 |\mathcal T(x)|\le\frac{r\varepsilon}{4}+2L|x-z|
 \le\left(\frac14+\frac{\sqrt2}{8}\right)r\varepsilon
 <r\varepsilon.
\]
Since $r\bar v=K^{-1}(h(c\mathbf e_1)-h(0))$, this proves the translation
alternative~\eqref{eq:jac-dichotomy-good} on the entire closed square.
\end{proof}

The proof uses only the upper Lipschitz bound and never divides by
$|\bar v|$. All vector estimates use the Euclidean inner product, so no
constant depends on $N$. The lower bilipschitz bound is used later,
in the finite-tree argument of Section~\ref{sec:jac-analytic}, to
bound the number of consecutive growth steps. The order of choice is
$L,\varepsilon$, then $M,\varphi,K_0$, then $k_0$, and finally $K$;
in particular neither $M$ nor $\varphi$ depends on $K$.

\section*{Acknowledgments}
The author thanks Professor Bobo Hua for his support.

\section*{Use of AI tools}
The author used OpenAI's Codex to investigate candidate constructions,
check calculations, and solicit separate AI reviews of draft proofs. The
author determined the direction of the work, verified all mathematical
claims and arguments, prepared the final manuscript, and assumes full
responsibility for its contents.

\bibliographystyle{amsplain-fullnames}
\makeatletter
\renewcommand{\@biblabel}[1]{\@defaultbiblabelstyle{#1}}
\renewcommand{\bibsetup}{}
\makeatother
\bibliography{bounded-geometry-counterexample}
\end{document}